\documentclass[reqno,11pt]{amsart}
\usepackage[all]{xypic}
\usepackage{enumitem,hyperref,bm,mathrsfs,amssymb,amsmath,amsfonts,amsthm,graphicx}

\newcommand{\dju}{\ensuremath{\mathaccent\cdot\cup}}
\newcommand{\G}{\mathscr{G}}
\newcommand{\partn}{\vdash}
\newcommand{\0}{\emptyset}
\newcommand{\qan}[1]{[{#1}]_q} 

\newcommand{\ol}{\overline}
\newcommand{\st}{~|~}

\newcommand{\Ff}{\mathbb{F}}
\newcommand{\Pp}{\mathbb{P}}
\newcommand{\Qq}{\mathbb{Q}}

\newcommand{\Nn}{\mathbb{N}}
\newcommand{\includefigure}[3]{
  \begin{center}
  \resizebox{#1}{#2}{\includegraphics{#3}}
  \end{center}}

\DeclareMathOperator{\gir}{gir}

\DeclareMathOperator{\mcd}{mcd}
\DeclareMathOperator{\nul}{nul}
\DeclareMathOperator{\rank}{rank}

\newcommand{\PS}{\mathcal{X}} 
\newcommand{\PSd}{\mathcal{X}^d} 
\newcommand{\PSI}{\mathcal{X}_I} 
\newcommand{\PSId}{\mathcal{X}^d_I} 
\newcommand{\PSD}{\mathcal{X}_D} 
\newcommand{\PV}{\mathcal{V}} 
\newcommand{\PVd}{\mathcal{V}^d} 

\newtheorem{theorem}{Theorem}
\numberwithin{theorem}{section}
\newtheorem{proposition}[theorem]{Proposition}
\newtheorem{corollary}[theorem]{Corollary}
\newtheorem{lemma}[theorem]{Lemma}

\theoremstyle{definition}
\newtheorem{definition}[theorem]{Definition}
\newtheorem{example}[theorem]{Example}
\newtheorem{remark}[theorem]{Remark}

\newcommand{\excise}[1]{}

\newcommand{\sm}{\setminus}

\title{Graph Varieties in High Dimension}
\author{Thomas Enkosky}
\address{Department of Mathematics,
University of Kansas,
Lawrence, KS 66047}
\email{tenkosky@math.ku.edu}
\author{Jeremy L.\ Martin}
\address{Department of Mathematics,
University of Kansas,
Lawrence, KS 66047}
\email{jmartin@math.ku.edu}
\date{June 20, 2011}
\thanks{Second author partially supported by an NSA Young Investigators Grant.}
\keywords{Graph variety, picture space, set partition, ear decomposition}
\subjclass[2000]{05C62, 14N20}
\begin{document}

\begin{abstract}
We study the \emph{picture space} $\PSd(G)$ of all embeddings of a finite graph $G$ as point-and-line arrangements in an arbitrary-dimensional projective space, continuing previous work on the planar case.  The picture space admits a natural decomposition into smooth quasiprojective subvarieties called \emph{cellules}, indexed by partitions of the vertex set of $G$, and the irreducible components of $\PSd(G)$ correspond to cellules that are maximal with respect to a partial order on partitions that is in general weaker than refinement.  We study both general properties of this partial order and its characterization for specific graphs.  Our results include complete combinatorial descriptions of the irreducible components of the picture spaces of complete graphs and complete multipartite graphs, for any ambient dimension~$d$.  In addition, we give two graph-theoretic formulas for the minimum ambient dimension in which the directions of edges in an embedding of $G$ are mutually constrained.
\end{abstract}

\maketitle

\section{Introduction}

The motivation for this paper is the geometric problem of determining the constraints on the directions of the lines formed by placing $n$ points in space and joining each pair of points with a line.

Let $G$ be a finite simple graph with vertices $V=V(G)$ and edges $E=E(G)$, let $\Ff$ be a field (typically algebraically closed), and let $\Pp^d=\Pp^d_\Ff$ denote projective $d$-dimensional space over~$\Ff$.  A \emph{picture} of $G$ consists of a point $P(v)$ for each vertex~$v$, and a line $P(e)$ for each edge $e$, all lying in~$\Pp^d$, subject to the conditions $P(v)\in P(e)$ whenever vertex~$v$ is an endpoint of edge~$e$.  The \emph{picture space} $\PS(G)=\PSd(G)$ of all pictures of~$G$ can be regarded naturally as a reduced scheme over~$\Ff$ (specifically, it is an algebraic subset of a product of Grassmannians, with the containment conditions expressible in terms of Pl\"ucker coordinates).

In general, the space $\PS(G)$ is not irreducible.  Its most important irreducible component, the \emph{picture variety} $\PV(G)$, consists of the Zariski closure of the locus of \emph{generic} pictures (i.e., those for which the points $P(v)$ are all different).  A picture in~$\PV(G)$ is called \emph{near-generic}; that is, it is either itself generic or is a limit of generic pictures.  Thus $\PV(G)=\PS(G)$ if and only if every picture is near-generic.  The projection of $\PV(G)$ onto the lines $\{P(e)\st e\in E\}$ (i.e., forgetting the locations of the points $P(v)$) is called the \emph{slope variety} of~$G$.

The picture space, picture variety and slope variety are collectively called \emph{graph varieties}; previous work of the second author on this subject includes \cite{GGV,Picspace,Slopes}.  A goal of this study is to understand the equations defining the picture variety as an irreducible component of the picture space; that is, the set of constraints on the directions of lines in a generic picture of~$G$.  In the case that $G$ is the complete graph on $n$ vertices, this is identical to the motivating problem stated at the start of this paper.

A related problem is to determine the component structure of the picture space.  Roughly speaking, each component other than $\PV(G)$ corresponds to a way in which the points of a picture can collapse, thus releasing some of the constraints on the directions of lines.  In the case of plane pictures (i.e., with ambient space $\Pp^2$), ``direction'' may be replaced with ``slope''.  The second author previously gave combinatorial descriptions of the components of the picture space, and of the generators of the ideal of slope constraints for an arbitrary graph~\cite{GGV}, and more detailed information about the Gr\"obner geometry of the picture space of the complete graph~\cite{Slopes}.  A key tool in these results is the theory of \emph{combinatorial rigidity} (for a good general reference, see \cite{GSS}).

The goal of this paper is to begin extending the theory of graph varieties to ambient dimension greater than~2.  Combinatorial rigidity is much less well understood in high dimension; on the other hand, some of the methods used in \cite{GGV} carry over effectively to the general setting, as we now explain.

The picture space has a natural decomposition into quasiprojective subvarieties called \emph{cellules}.  For each set partition $\pi$ of $V(G)$, the cellule $\PS_\pi(G)=\PSd_\pi(G)$ consists of the pictures for which $P(v)=P(w)$ if and only if $v,w$ belong to the same block of $\pi$.  For instance, the locus of generic pictures is the cellule corresponding to the partition of $V(G)$ into singleton subsets.  Every irreducible component of $\PS(G)$ is in fact the Zariski closure of some cellule (Proposition~\ref{cpt-cellule-closure}), although not every cellule closure is a full component (because some cellules can be contained in the closures of other cellules).  Accordingly, we define the \emph{cellule (partial) order} $\prec_{G,d}$ on set partitions $\pi,\sigma$ by
\begin{equation}
\label{partial-order}
\pi\prec_{G,d}\sigma\iff \PSd_\pi(G)\subseteq\overline{\PSd_\sigma(G)}.
\end{equation}
Therefore, describing the components of $\PSd(G)$ reduces to the combinatorial problem of determining the maximal set partitions with respect to the cellule order.

By way of motivation, consider the case $d=1$.  A one-dimensional picture of~$G$ is just an ordered $n$-tuple of points in $\Pp^1$, where $n=|V|$; the data corresponding to edges of $G$ is trivial.  Thus $\PS^1(G)=(\Pp^1)^n$, and the cellules are just the elements of the intersection lattice of the projectivized braid arrangement (see, e.g.,~\cite[pp.~8--9]{HA}), partially ordered by refinement: $\pi\prec_{G,d}\sigma$ if every block of~$\sigma$ is contained in some block of~$\pi$.  In the case $d=2$, some refinement relations disappear; for the precise statement, which uses rigidity theory, see \cite[Thm.~6.3]{GGV}.  As the ambient dimension~$d$ increases, the cellule order contains fewer and fewer relations.

We start by establishing the following general facts about the cellule order.  First, let $B$ be a set of vertices that induces an acyclic subgraph, and $\pi$ is any set partition containing $B$ as a block.  Then the points $\{P(v)\st v\in B\}$ of $B$ in a picture $P\in\PS_\pi(G)$ can be separated: that is, $P$ can be obtained as a limit of pictures in a cellule indexed by a refinement $\sigma$ of $\pi$.  Therefore, $\pi\prec_{G,d}\sigma$ (Corollary~\ref{acyclic-corollary}).  In particular, if $G$ itself is acyclic, then $\PS(G)$ is irreducible, and in fact smooth~\cite[\S7]{Picspace}.  An immediate application is a description of the  components of the picture space of the cycle $C_n$: if $d<n$ then $\PSd(C_n)$ is irreducible, while if $d\geq n$ then it has two components, namely the picture variety and the indiscrete cellule (the locus of pictures in which all points $P(v)$ coincide).  This component structure corresponds to the following simple geometric fact: in a generic picture of~$C_n$, the lines corresponding to the edges connect $n$ distinct points, so their affine span has dimension at most $n-1$.  If $d<n$ then this constraint is vacuous; on the other hand, if $d\geq n$, then the affine span of the lines in an indiscrete picture can have dimension~$n$, so not every picture is near-generic.

On the other hand, in order to show that breaking up a block $B$ does \emph{not} correspond to a relation in $\prec_{G,d}$, it often suffices to calculate the dimension of the indiscrete cellule in $\PS(G|_B)$ (Proposition~\ref{squash-a-block}).  This suggests a method of determining the component structure for families of graphs that are \emph{hereditary}, i.e., closed under taking induced subgraphs.  As an application, consider the \emph{complete multipartite graph} $G=K_{q_1,\dots,q_n}$, which consists of $q_i$ vertices of color~$i$ for each $i$, $1\leq i\leq n$, with an edge between each pair of vertices of different colors.  The general results described above lead to a explicit, if somewhat technical, combinatorial description of the components of~$\PSd(G)$ (Theorem~\ref{Kmn-components}).  In the important special case of the complete graph~$K_n$ (which can be regarded as a complete multipartite graph with one vertex of each color), the component structure becomes much easier to describe.

\textbf{Theorem~\ref{Kn-components}:} Let $d\geq 3$.  Then the components of the picture space $\PSd(K_n)$ are exactly the closures of cellules corresponding to set partitions with no block of size two.

We find it notable that the result is the same for all ambient dimensions $d\geq 3$, and is much simpler than the $d=2$ case \cite[Thm.~6.4]{GGV}.  A rough combinatorial explanation is that for $d=2$, the minimal obstructions to irreducibility are rigidity circuits, while when $d\geq 3$, the minimal obstructions are cycles of length~$d$ (although this is not quite true; see below) which are combinatorially easier to describe.

As observed above, the cellule order becomes weaker and weaker as $d$ increases.  Accordingly, we define the \emph{minimum constraint dimension} $\mcd(G)$ to be the smallest positive integer~$d$ for which $\PSd(G)$ is not irreducible (or $\infty$ if no such $d$ exists).  Equivalently, $\mcd(G)$ is the smallest ambient dimension in which the directions of the lines in a generic $d$-dimensional picture of $G$ admit some mutual constraint.  For instance, as discussed above, the $n$-cycle $C_n$ has $\mcd(C_n)=n$; more generally, $\mcd(G)\leq\gir(G)$, where $\gir(G)$ denotes the \emph{girth} of $G$, that is, the length of the smallest cycle.  This bound is not sharp; for instance, the graph obtained by identifying two 4-cycles along an edge has girth~4, but minimum constraint dimension~3.  In fact, a result of B.~Servatius~\cite{Ser} implies that the gap between $\gir(G)$ and $\mcd(G)$ can be arbitrarily wide; see Remark~\ref{girth-mcd} below.

Previous work of the second author \cite[Theorem~14]{Picspace} implies that the minimum constraint dimension can be determined from the Tutte polynomial of $G$.  (For general information on the Tutte polynomial, see \cite{BryOx}.)  We give two simpler versions of this formula, respectively in terms of the graphic matroid of $G$ (Proposition~\ref{mcd-formula}) and ear decompositions (Proposition~\ref{ear-formula}).  As a more general class of examples, we calculate the minimum constraint dimension for the ``onion'' graph formed by identifying multiple paths at their endpoints.

\section{Components of the Picture Space} \label{component-section}

Fix an ambient dimension $d\geq 3$ and a graph $G=(V,E)$.  The notation $\pi\partn V$ means that $\pi$ is a partition of $V$ into pairwise disjoint subsets, called \emph{blocks}.  The corresponding equivalence relation will be denoted by $\sim_\pi$; thus $v\sim_\pi w$ if and only if $v,w$ belong to the same block of $\pi$.  The number of blocks of $\pi$ will be denoted $|\pi|$.

The \emph{cellule} corresponding to $\pi$ is defined as
\begin{equation}
\label{define-cellule}
\PSd_\pi(G) = \left\{P\in\PSd(G) ~\Big|~ P(v)=P(w)\iff v\sim_\pi w\right\}.
\end{equation}
Thus the picture space is the disjoint union of the cellules.
Moreover, each cellule is a smooth, quasiprojective subvariety
of $\PSd(G)$ with a natural bundle structure as follows.  Consider the
projection map $\PS_\pi(G)\to(\Pp^d)^{|\pi|}$ that records only the positions of the points (one for each block~$B_i$).  This makes $\PS_\pi(G)$ into a bundle with smooth base
  $$\left\{(p_1,\dots,p_{|\pi|})\in(\Pp^d)^{|\pi|} ~\Big|~ p_i\neq p_j \text{ for all $i\neq j$}\right\}$$
 and smooth fiber $(\Pp^{d-1})^{\delta(\pi,G)}$,
where $\delta(\pi,G)$ is the number of edges with both endpoints in
the same block of~$\pi$.  Therefore,
  \begin{equation}
  \label{dim-cellule}
  \dim \PS_\pi(G) = d|\pi|+(d-1)\delta(\pi,G).
  \end{equation}
Important special cases are the \emph{discrete cellule} $\PSD(G)$ and the \emph{indiscrete cellule} $\PSI(G)$, which correspond respectively to the \emph{discrete partition} of $V$ into singleton blocks, and the \emph{indiscrete partition} with only one block.  Thus the discrete cellule $\PSD(G)$ consists precisely of the generic pictures, and its closure is the picture variety $\PV(G)$, which is the whole picture space if and only if $\PSD(G)$ is the unique cellule of largest dimension~\cite[Thm.~4.5]{GGV}.  Meanwhile, the indiscrete cellule $\PSI(G)$ has a smooth bundle structure, with base $\Pp^d$ and fiber $(\Pp^{d-1})^{|E(G)|}$; it is the only cellule that is Zariski-closed in $\PS(G)$.  Note that
$\dim\PSD(G)=d|V|$ and $\dim\PSI(G)=d+(d-1)|E|$.

\begin{proposition}
\label{cpt-cellule-closure}
For every graph $G$ and ambient dimension $d$,
the components of $\PSd(G)$ are exactly the maximal sets of the form
$\overline{\PS_\pi}$.
\end{proposition}

\begin{proof}
Each cellule is irreducible and smooth (by its description
above as a bundle); therefore, its closure is irreducible
\cite[Example~1.1.4, p.~3]{Har}.  Therefore,
the picture space can be written as the union of finitely many
closed, irreducible subsets, namely the closures of the cellules,
so each irreducible component must appear in this union.
\end{proof}

Accordingly, to understand the component structure of $\PSd(G)$,
it suffices to describe the cellule order $\prec_{G,d}$ combinatorially.
This problem is the subject of the remainder of this section.

\begin{lemma}
Let $\sigma,\pi\partn V(G)$ with $\pi\prec_{G,d}\sigma$.
Then $\sigma$ refines $\pi$.
\end{lemma}

\begin{proof}
If $\sigma$ does not refine $\pi$, then there is some pair of vertices $x,y$ such that $x,y$ are in the same block of $\sigma$, but not in the same block of $\pi$.  The equality $P(x)=P(y)$ is a Zariski-closed condition that holds on $\PS_\sigma(G)$, hence on $\overline{\PS_\sigma(G)}$, but not on $\PS_\pi(G)$.  (In fact, we have shown more: if $\sigma$ does not refine $\pi$, then $\overline{\PS_\sigma(G)}\cap \PS_\pi(G)=\emptyset$.)
\end{proof}

\begin{proposition}
\label{split-doubleton}
Let $\pi,\sigma$ be partitions of $V(G)$ such that $\sigma$ is obtained from $\pi$ by splitting a doubleton block into two singletons, say $\pi=\{B_1,\dots,B_r,\{x,y\}\}$ and $\sigma=\{B_1,\dots,B_r,\{x\},\{y\}\}$.  Then $\PS_\pi(G)\subset\overline{\PS_\sigma(G)}$.  In particular, no partition containing a doubleton block corresponds to a maximal cellule.
\end{proposition}

\begin{proof}
Fix a picture $P_0\in\PS_\pi(G)$.  If $x,y$ are joined by an edge $e$, then let $L=P_0(e)$; otherwise, let~$L$ be any line containing the point~$P_0(x)=P_0(y)$.  Let $Z$ be the locus of all pictures $P$ such that
\begin{itemize}
\item $P(v)=P_0(v)$ for all vertices $v\neq y$;
\item $P(y)\in L$; and
\item $P(f)=P_0(f)$ for all edges $f$ with both endpoints in the same block of $\pi$.
\end{itemize}
Note that if $f$ is an edge with endpoints in different blocks
of $\pi$, then $P_0(f)$ is determined by the foregoing data,
with finitely many exceptions (namely, if $P(y)=P(v)$
for some $v\in V\sm\{x,y\}$).  Therefore, the map $P\mapsto P(y)$
is a birational equivalence of $Z$ with $L$ itself.

In particular, setting $P(y)=P_0(y)$ gives $P=P_0$.  It follows
that $P_0\in \overline{\PS_\sigma(G)}$.  Since $P_0$ was chosen
arbitrarily in $\PS_\pi(G)$, we conclude that
$\PS_\pi(G)\subset\overline{\PS_\sigma(G)}$.
\end{proof}

Note that the proof of Proposition~\ref{split-doubleton} fails for larger blocks.  Pulling apart the vertices in a block of size $r>2$ requires the corresponding $\binom{r}{2}$ lines in $P$ to lie in a common $(r-1)$-dimensional space, which need not be the case for all pictures.  On the other hand, a stronger result using essentially the same argument is as follows:

\begin{proposition}
\label{general-split-doubleton}
Let $G=(V,E)$ be any graph (not necessarily simple),
let $S\subseteq V$ and let $y\in V\sm S$
be a vertex such that no more than one edge incident
to $y$ has its other endpoint in $S$.  Let $S'=S\cup\{y\}$,
let $\pi$ be a partition of $V$ such that $S'$ is a block of
$\pi$, and let $\sigma=\pi\sm\{S'\}\cup\{S,\{y\}\}$.
Then $\PS_\pi(G)\subset\overline{\PS_\sigma(G)}$.
\end{proposition}

\begin{proof}
Let $P_0\in\PS_\pi(G)$.  If $G$ has an edge $e$ between $y$ and a vertex in $S$, then let $L=P_0(e)$, otherwise, let~$L$ be any line through the point $P_0(y)$.  As before, construct a family of pictures $P$ such that $P(x)=P_0(x)$ for $x\in V\sm\{y\}$; $P(y)$ varies along $L$; $P(e)=P_0(e)$ (if applicable), and $P(f)=P_0(f)$ for all edges $f$ with both endpoints in the same block of $\pi$.  Then all but finitely many $P$ lie in $\PS_\sigma(G)$, and setting $P(y)=P_0(y)$ gives $P=P_0$, so $P_0\in\overline{\PS_\pi(G)}$.
\end{proof}

\begin{corollary} \label{acyclic-corollary}
Let $W\subseteq V$ such that the induced subgraph $G|_W$ is acyclic.
Let $\sigma,\pi\partn V$ such that $W$ is a block of $\pi$ and $\sigma$ is
obtained by $\pi$ from subdividing $W$ into smaller blocks.  Then
$\PS_\pi(G)\subseteq\overline{\PS_\sigma(G)}$.
\end{corollary}

An immediate application of Corollary~\ref{acyclic-corollary} is a simple description of the components of the picture space of a cycle.

\begin{corollary} \label{cycle-corollary}
Let $C_n$ be the cycle on $n$ vertices.  If $d<n$, then $\PS^d(C_n)$ is irreducible, while if
$d\geq n$, then $\PS^d(C_n)$ has two components, namely the picture variety $\PVd(C_n)$
and the indiscrete cellule $\PSId(C_n)$.
\end{corollary}

\begin{proof}
By Corollary~\ref{acyclic-corollary}, every cellule other than $\PSId(C_n)$ is contained in the picture variety, so $\PSId(C_n)$ is the only other possible component.  Meanwhile, the cellule-dimension formula~\eqref{dim-cellule} implies that $\dim\PSId(G)\geq\dim\PVd(G)$ if and only if $d\geq n$;
as mentioned above, we have $\PSd(G)=\PVd(G)$ if and only if $\PSD(G)$ is the unique cellule of largest dimension~\cite[Thm.~4.5]{GGV}, implying the result.
\end{proof}

Indeed, when $d\geq n$, the picture space of $C_n$ is not irreducible for the geometric reason discussed in the Introduction --- the lines of a generic picture of $C_n$ must span an affine space of dimension at most $n-1$, while this constraint does not apply to pictures in the indiscrete cellule.
In this way, cycles behave in ambient dimension $\geq 3$ just as rigidity circuits do in ambient dimension~2 (cf.~\cite[Lemma~6.2]{GGV}).

We now turn to the problem of finding sufficient conditions for a cellule to be maximal in the cellule order.

\begin{definition}
Let $d\geq 2$.  A graph $H$ is called \emph{$d$-heavy} if the indiscrete cellule $\PSId(H)$ has maximum dimension (not necessarily uniquely) among all cellules of $\PSd(H)$.
\end{definition}

If $H$ is $d$-heavy, then the indiscrete cellule is maximal in the cellule order, hence is a component.  The following result reduces the geometric statement that a cellule $\PSd_\pi(G)$ is maximal to the combinatorial statement that the induced subgraph $G|_B$ is $d$-heavy for every block $B\in\pi$.

\begin{proposition}\label{squash-a-block}
Let $G$ be a graph and~$d\geq 2$.  Let~$B$ be a set of vertices of~$G$
such that the induced subgraph~$H=G|_B$ is $d$-heavy.
Let $\sigma,\pi$ be partitions of~$V(G)$ such that $\sigma$ refines~$\pi$
and $B\in\pi\sm\sigma$. Then $\PSd_\pi(G)\not\subset\overline{\PS_\sigma^d(G)}$.
\end{proposition}

\begin{proof}
Suppose that $B=B_1\cup\cdots\cup B_r$, where $B_1,\dots,B_r\in\sigma$.
Let $\phi$ be the map that takes a picture of $G$ and forgets all vertices
except those in~$B$, and all edges except those with both endpoints in~$B$.
We have a commutative diagram

\begin{equation}
\xymatrix{
 Z=\PS_\pi(G) \rto\dto^{\phi} & \PS(G) \dto^{\phi} & \PS_\sigma(G)=Y \dto^{\phi}\lto \\
\PSI(H) \rto & \PS(H) & \PS_{\tilde\sigma}(H)\lto
}
\end{equation}
where $\tilde\sigma=\{B_1,\dots,B_r\}$.
Here each vertical map is a surjection, and each horizontal
map is the natural inclusion.
Now, suppose that $\bar Y\supset Z$.  Then 
$\overline{\phi(Y)}=\overline{\phi(\overline
Y)}\supset\overline{\phi(Z)}\supset\phi(Z)$.
But this contradicts
the hypothesis that $H$ is $d$-heavy.
\end{proof}

The complete multipartite graph $G=K_{q_1,\dots,q_n}$ has $q_i$ vertices of color~$i$ for each $i\in[n]=\{1,2,\dots,n\}$, with an edge between every pair of vertices of different colors (for short, a \emph{heterochromatic pair}).  For convenience, we assume that $q_1\geq\cdots\geq q_n>0$.  (Note that here $n$ denotes the number of colors, not the number of vertices.)  Proposition~\ref{squash-a-block} can be used as a general tool to characterize the component structures of picture spaces of complete and complete multipartite graphs.  Our methods may apply more generally to \emph{hereditary} familes of $d$-heavy graphs: a family $\G$ is hereditary if every induced subgraph of a member of~$\G$ is also a member of~$\G$.

We assume in what follows that either $n\geq 3$, or if $n=2$ then $q_1,q_2\geq 2$.  (Otherwise, $G$ is acyclic.)  Let $V_i$ be the set of vertices of color $i$ and let $V=V_1\cup\cdots\cup V_n$.  Each set partition
$\pi\partn V$ with $r$ blocks can be written as
$$
\pi ~=~ \{B_1 = B_{11}\cup\cdots\cup B_{1n},\ \ \dots,\ \
B_r = B_{r1}\cup\cdots\cup B_{rn}\}
$$
where $V_j=B_{1j}\dju\cdots\dju B_{rj}$ for each $j\in[n]$.  The dimension of the cellule $\PSd_\pi(G)$ depends only on the numbers $b_{ij}=|B_{ij}|$, so it is convenient to regard $\pi$ as a way of placing colored balls in boxes, with the $i^{th}$ box containing $b_{ij}$ balls of color~$j$ for each~$i,j$, and no box empty.  The cellule dimension formula \eqref{dim-cellule} may be rephrased in terms of balls and boxes:
\begin{equation} \label{ballbox}
\dim\PSd_\pi(G) = d|\{\text{boxes}\}| + (d-1)|\{\text{heterochromatic pairs in the same box}\}|.
\end{equation}

\begin{proposition} \label{complete-multip-degen}
Let $d\geq 3$ and let $G=K_{q_1,\dots,q_n}$ be a complete multipartite graph on $n\geq 2$ colors; if $n=2$, then assume that $q_1\geq 3$ and $q_2\geq 2$.  Then $G$ is $d$-heavy.  In particular, the indiscrete cellule of $\PSd(G)$ is an irreducible component.
\end{proposition}

\begin{proof}
Consider any set partition $\pi\partn V(G)$, corresponding to some way of placing colored balls in boxes as described above.  We will show that it is possible to  merge all the boxes into a single box, step by step, with each step either preserving or increasing the dimension of the corresponding cellule.  It will follow that the indiscrete cellule has dimension greater than or equal to every other cellule.  The following rules list types of mergers with this property.  (These rules are of course invariant with respect to permuting colors; it is more convenient to say, e.g., ``two red balls and one orange ball'' rather than ``two balls of the same color and a third ball of a different color''.)

\begin{enumerate}[label=(\roman{*}), ref=(\roman{*})]
\item\label{RR-O} Merge $B_1$ and $B_2$, where $B_1$ contains two red balls and $B_2$ contains an orange ball;
\item\label{RO-RO} Merge $B_1$ and $B_2$, where $B_1$ and $B_2$ each contain a red ball and a orange ball;
\item\label{RO-G} Merge $B_1$ and $B_2$, where $B_1$ contains a red ball and an orange ball, and $B_2$ contains a green ball.
\item\label{RO-X} Merge $B_1$ and $B_2$, where $B_1$ contains balls of at least three colors, and $B_2$ is arbitrary.
\item\label{R-O-G} Merge $B_1$, $B_2$ and $B_3$, where $B_1$ contains a red ball, $B_2$ contains an orange ball, and $B_3$ contains a green ball;
\item\label{RO-R-O} Merge $B_1$, $B_2$ and $B_3$, where $B_1$ contains a red and an orange ball, $B_2$ contains a red ball, and $B_3$ contains an orange ball.
\item\label{R-R-R-O-O} Merge $B_1,\dots,B_5$, where $B_1,B_2,B_3$ each contain a red ball and $B_4,B_5$ each contains an orange ball.
\end{enumerate}

Checking that the cellule dimension remains the same or increases after any of these mergers is an elementary consequence of \eqref{ballbox}.  For example, each of the mergers~\ref{RR-O}\dots\ref{RO-X} results in one fewer box, but at least two additional heterochromatic pairs in the same box, hence increases dimension by at least $-d+2(d-1)=d-2>0$; merger~\ref{R-R-R-O-O} increases dimension by at least $-4d+6(d-1)=2d-6\geq 0$.

We now show how to use the merging rules to obtain the indiscrete partition from an arbitrary partition~$\pi$.

\begin{itemize}
\item \textbf{Case 1:} $n\geq 3$.

\begin{itemize}
\item \textit{Case 1a:} If some box of $\pi$ contains two or more red balls, then either it contains all the orange balls or, by~\ref{RR-O}, can be merged with a box containing an orange ball.  The resulting box either contains all the green balls, or, by~\ref{RR-O}, can be merged with a box containing a green ball.  The resulting box contains balls of at least three colors, and so can be merged with every other box in turn by repeated applications of~\ref{RO-X}.

\item \textit{Case 1b:} If some box contains both a red ball and an orange ball, then either it contains all the green balls, or, by~\ref{RO-G}, can be merged with a box containing a green ball.  The resulting box contains balls of at least three colors, and so can be merged with every other box.

\item \textit{Case 1c:} The only other possibility is that $\pi$ is the discrete partition.  By~\ref{R-O-G}, we can merge three boxes containing balls of different colors, then proceed as in Case~1a.
\end{itemize}

\item \textbf{Case 2:} $n=2$.  Recall that we have assumed that $q_1\geq 3$ and $q_2\geq 2$.

\begin{itemize}

\item \textit{Case 2a:} If some box of $\pi$ contains two red balls, then by~\ref{RR-O} it can be merged with all boxes containing orange balls, then with all the other boxes.

\item \textit{Case 2b:} If two boxes each contain both a red and an orange ball, then we can first merge them by~\ref{RO-RO}, then proceed as in Case~2a.

\item \textit{Case 2c:} If one box contains two balls, one red and one orange, and all other boxes are singletons, then we can merge the doubleton box with two singleton boxes with different-colored balls by~\ref{RO-R-O}, then proceed as in Case~2a.

\item \textit{Case 2d:} If all boxes contain singletons, then we can first apply~\ref{R-R-R-O-O}, then proceed as in Case~2a.
\end{itemize}
\end{itemize}

In all cases, we can eventually merge all boxes into a single box, while increasing or preserving dimension at every stage of the process.
\end{proof}

It is possible to relax the restrictions on $d$ and $q_1,\dots,q_n$ in the statement of Proposition~\ref{complete-multip-degen}.  We have not done so, because if $d=2$, then the component structure of $\PSd(G)$ is already known~\cite[\S6]{GGV}, while $K_{2,2}$ is just a 4-cycle, so its component structure is given by Corollary~\ref{cycle-corollary}.

\begin{theorem} \label{Kmn-components}
Let $d\geq 3$, and let $G=K_{q_1,\dots,q_n}$ be a complete multipartite graph other than an acyclic graph or $K_{2,2}$.  Then the components of $\PSd(G)$ are exactly the cellule closures $\ol{\PSd_\sigma(G)}$, where $\sigma$ is a set partition of~$[n]$ in which every block either (i) contains only one vertex; (ii) contains at least two red vertices and at least two orange vertices; or
(iii) contains vertices of three or more colors.
\end{theorem}

\begin{proof}
Suppose that $\sigma$ contains a block that is not of any of these forms.  Then the induced subgraph on that block is acyclic, and by Corollary~\ref{acyclic-corollary}, the cellule $\PSd_\sigma(G)$ is not maximal.  On the other hand, if $\sigma$ meets the conditions of the theorem, then $\sigma$ is maximal in cellule order by Proposition~\ref{complete-multip-degen}.  Therefore, by Proposition~\ref{cpt-cellule-closure}, $\overline{\PSd_\sigma(G)}$ is a component of $\PSd(G)$, and these are all the components.
\end{proof}

An important special case is the complete graph $K_n$, which is the complete multipartite graph $K_{1,1,\dots,1}$.  If $n=2$ then $\PSd(K_n)$ is smooth; otherwise, since there is only one vertex of each color, condition~(ii) of Theorem~\ref{Kmn-components} is irrelevant, and conditions (i) and (iii) say that each block must be either a singleton, or else contain at least three vertices.  That is:

\begin{theorem} \label{Kn-components}
Let $d\geq 3$ and $n\geq 2$.  Then the components of $\PSd(K_n)$ are exactly the cellule closures $\ol{\PS_\sigma(K_n)}$, where $\sigma$ ranges over all set partitions of~$[n]$ with no blocks of size two.
\end{theorem}

We close this section with some potential problems for future study.  Proposition~\ref{complete-multip-degen} is far from sharp, in the sense that there are $d$-heavy graphs with many fewer edges than a complete multipartite graph.  This suggests looking for a lower bound on the number of edges that guarantees $d$-heaviness.  Another possibility is to look more closely at the cellule order.  For instance, is it possible to characterize all posets arising as $\prec_{G,d}$ for some $G$ and $d$?  (Every such poset must of course be a weakening of the partition lattice.)  In addition, what can be said about the order ideal generated by the discrete partition, i.e., the set of cellules consisting of quasi-generic pictures?

\section{Minimum constraint dimension as a combinatorial invariant}
\label{mcd-section}
Let $G$ be a connected graph that is not acyclic.  We define the \emph{minimum constraint dimension} $\mcd(G)$ to be the smallest positive integer~$d$ for which $\PSd(G)$ is not irreducible.  Implicit in the results of \cite{Picspace} is that $\mcd(G)$ can be calculated by reading off information about the irreducible components of $\PS(G)$ from the Tutte polynomial of $G$, as we now explain.

The Tutte polynomial has the well-known \emph{corank-nullity formula}
$$T_G(x,y) = \sum_{A\subseteq E} (x-1)^{r(E)-r(A)} (y-1)^{\nul(A)}$$
where $r(A)$ denotes the rank of $A$ (that is, the size of a maximum acyclic subset of~$A$) and $\nul(A)=|A|-r(A)$.  By \cite[Thm.~1]{Picspace}, the homology of $\PSd(G)$ is free abelian and concentrated in even real dimension, and the Poincar\'e polynomial of $\PSd(G)$ --- the generating function for its topological Betti numbers --- is given by the formula
  \begin{equation}\label{poinc}
  \sum_{i\geq0} \ q^{2i} \dim H_i(\PSd(G),\Qq) \ = \ 
  (\qan{d}-1)^{|V|-1}\qan{d+1}T_G\left(\frac{\qan{2}\qan{d}}{\qan{d}-1},\qan{d}\right)
  \end{equation}
where $\qan{d}=1+q+q^2+\cdots+q^{d-1}=(1-q^d)/(1-q)$.  In particular, $\PSd(G)$ is irreducible if and only if the Poincar\'e polynomial is monic of degree $d|V|$.  Comparing the leading terms of the summands (for details, see \cite[Prop.~3.3]{DMR}) implies that $\PSd(G)$ is irreducible if and only if $d\cdot\nul(A)<|A|$ for all nonempty $A\subseteq E$.  Therefore,
\begin{align}
\label{irred-condition}
\mcd(G) &= \min\{d\in\Nn \st d\cdot\nul(A)\geq|A|\ \text{ for some nonempty } A\subseteq E\}\\
\label{mcd-over-A}
&= \min_{\0\neq A\subseteq E}\left\lceil\frac{|A|}{\nul(A)}\right\rceil.
\end{align}
In particular, the condition \eqref{irred-condition} becomes stricter as $d$ increases, so in fact $\PSd(G)$ is irreducible if and only if $d<\mcd(G)$.

A consequence of \eqref{irred-condition} is the inequality $\mcd(G)\leq\gir(G)$, where $\gir(G)$ denotes the \emph{girth} of $G$, i.e., the size of the smallest cycle in $G$.  Geometrically, the lines corresponding to the edges of an $n$-cycle must all lie in some affine space of dimension $<n$, giving a nontrivial constraint on their directions when $d\geq n$; combinatorially, if $A\subseteq E$ is the edge set of a cycle, then $r(A)=|A|-1$, and so \eqref{irred-condition} fails for $d\geq |A|$.  On the other hand, cycles with common edges can interact to produce a tighter constraint on $\mcd(G)$,
as in the case of ``onion graphs''; see Remark~\ref{girth-mcd} below.  A rough explanation is that contracting the edge set of a cycle can decrease the girth of a graph; this observation suggests interpreting $\mcd(G)$ in terms not only of cycles in~$G$, but also cycles of graphs produced from~$G$ by repeatedly contracting cycles.

We can restrict the edge sets that need to be considered in \eqref{mcd-over-A}.
In what follows, we assume familiarity with basic facts about graphic matroids (see, e.g., \cite{Oxley}, particularly \S4).

\begin{proposition}
\label{mcd-formula}
For every connected, non-acyclic graph $G$, we have
$$\mcd(G) = \min_A\left\lceil\frac{|A|}{\nul(A)}\right\rceil$$
where $A$ ranges over all nonempty flats of the graphic matroid $M=M(G)$ such that $M|_A$ is indecomposable as a direct sum.  (Graph-theoretically, $A$ ranges over all nonempty edge sets of 2-connected induced subgraphs of~$G$.)
\end{proposition}

\begin{proof}
Let $A$ be a nonempty subset of $E$ minimizing $|A|/\nul(A)$.
If $M|_A=M|_{A_1}\oplus M|_{A_2}$, then $|A|=|A_1|+|A_2|$ and $\nul(A)=\nul(A_1)+\nul(A_2)$, so there is some $i$ for which $|A_i|/\nul(A_i)\leq|A|/\nul(A)$.  Therefore, we may assume that $M|_A$ is indecomposable.  Meanwhile, $|A|/\nul(A)=(\nul(A)+\rank(A))/\nul(A)=1+\rank(A)/\nul(A)$.  If $\bar A$ denotes the closure of $A$ in $M$, then $\rank(\bar A)=\rank(A)$ and $\nul(\bar A)\geq\nul(A)$; therefore, $|\bar A|/\nul(\bar A)\leq|A|/\nul(A)$.  Therefore, we may assume that $A=\bar A$, i.e., that $A$ is a flat of $M$.
\end{proof}

The formula can be restated in more explicitly graph-theoretic terms.  A \emph{partial ear decomposition} of $G$ is a sequence $(C_1,\dots,C_k)$ of pairwise disjoint edge sets such that $C_i$ is an induced cycle of $G/C_1/\cdots/C_{i-1}$ for each $i$.

\begin{proposition} \label{ear-formula}
For every connected, non-acyclic graph $G$, we have
\begin{equation}
\mcd(G) = \min_{(C_1,\dots,C_k)}\left\lceil\frac{\sum_{i=1}^k|C_i|}{k}\right\rceil
\end{equation}
the minimum over all partial ear decompositions of $G$.
\end{proposition}

\begin{proof}
Every 2-connected edge set $A$ (in fact, every 2-edge-connected graph) has an ear decomposition~\cite[Prop~2.10]{Frank}, so the result follows from replacing the set $A$ in Proposition~\ref{mcd-formula} with the partial ear decomposition $(C_1,\dots,C_k)$ and observing that $k=\nul(A)$.
\end{proof}

\begin{example}
\label{onion}
Let $a_1\leq\cdots\leq a_k$ be positive integers, and let $O=O(a_1,\dots,a_k)$ 
be the ```onion'' graph formed by identifying $k$ disjoint paths $P_1,\dots,P_k$
of lengths $a_1,\dots,a_k$ at their endpoints, as shown.
\includefigure{3in}{1.5in}{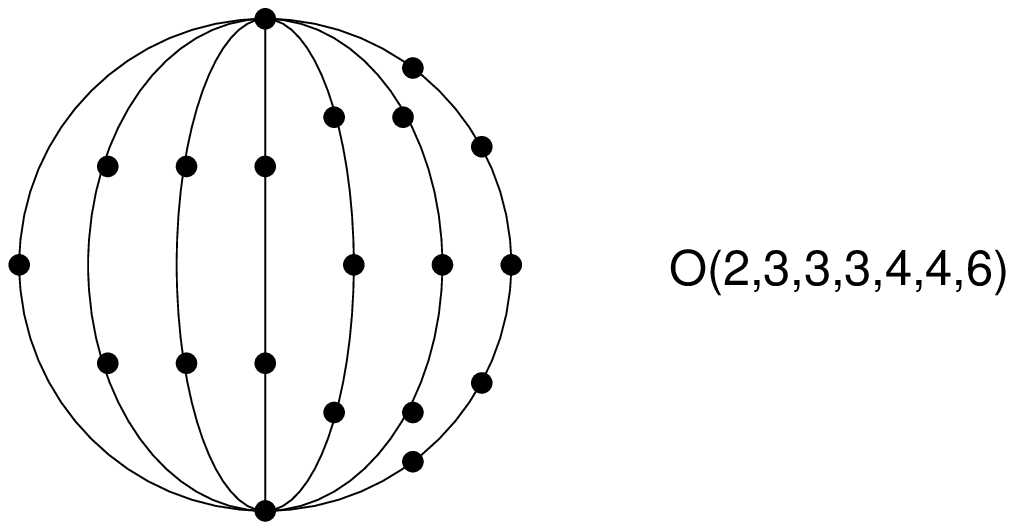}
The 2-connected edge sets of $O$ are precisely the unions of two or more of
the $P_i$, so
$$\mcd(O) = \left\lceil\min_{2\leq r\leq k}\left(\frac{a_{1}+\cdots+a_{r}}{r-1}\right)\right\rceil.$$
On the other hand, the girth of $O$ is $a_1+a_2$, which can be considerably larger;
for instance, if $k\gg 0$ and $a_1=\cdots=a_k=a$, then $\mcd(O)=a$ and $\gir(O)=2a$.
\end{example}

\begin{remark}
\label{girth-mcd}
The ratio $\gir(G)/\mcd(G)$ can be arbitrarily high, at least when $\mcd(G)=2$.  By a theorem of B.~Servatius \cite{Ser}, there exist 2-rigid graphs of arbitrarily high girth.  For any positive integer~$\ell$, let $G$ be a 2-rigid graph with girth $\geq 2\ell-2$, let $C$ be a minimum-length cycle in~$G$, and let~$x,y$ be vertices of $C$ at maximum distance, so that every path in~$G$ between~$x,y$ has length $\geq \ell-1$.  Then $G+xy$ contains a 2-rigidity circuit~$G'$, which has $\mcd(G')=2$; on the other hand, $\gir(G')\geq\gir(G)\geq\ell$.
\end{remark}

\begin{remark}
\label{edmonds}
The inequality in \eqref{irred-condition} can be rewritten as $|A|\geq(d/(d-1))r(A)$, which closely resembles Edmonds' condition~\cite{Edmonds} for decomposability of a matroid into independent sets; when $d=2$, it is precisely Laman's characterization~\cite{Laman} of 2-rigidity independence.  For more on this connection, see \cite{DMR}.
\end{remark}


\bibliographystyle{amsalpha}
\bibliography{biblio.bib}

\providecommand{\bysame}{\leavevmode\hbox to3em{\hrulefill}\thinspace}
\providecommand{\MR}{\relax\ifhmode\unskip\space\fi MR }
\providecommand{\MRhref}[2]{%
  \href{http://www.ams.org/mathscinet-getitem?mr=#1}{#2}
}
\providecommand{\href}[2]{#2}
\begin{thebibliography}{DMR07}

\bibitem[BO92]{BryOx}
Thomas Brylawski and James Oxley, \emph{The {T}utte polynomial and its
  applications}, Matroid applications, Encyclopedia Math. Appl., vol.~40,
  Cambridge Univ. Press, Cambridge, 1992, pp.~123--225. \MR{93k:05060}

\bibitem[DMR07]{DMR}
Mike Develin, Jeremy~L. Martin, and Victor Reiner, \emph{Rigidity theory for
  matroids}, Comment. Math. Helv. \textbf{82} (2007), no.~1, 197--233.
  \MR{MR2296062 (2008h:52026)}

\bibitem[Edm65]{Edmonds}
Jack Edmonds, \emph{Minimum partition of a matroid into independent subsets},
  J. Res. Nat. Bur. Standards Sect. B \textbf{69B} (1965), 67--72. \MR{32
  \#7441}

\bibitem[Fra95]{Frank}
Andr{\'a}s Frank, \emph{Connectivity and network flows}, Handbook of
  combinatorics, {V}ol.\ 1,\ 2, Elsevier, Amsterdam, 1995, pp.~111--177.
  \MR{MR1373657 (97i:05084)}

\bibitem[GSS93]{GSS}
Jack Graver, Brigitte Servatius, and Herman Servatius, \emph{Combinatorial
  rigidity}, Graduate Studies in Mathematics, vol.~2, American Mathematical
  Society, Providence, RI, 1993. \MR{95b:52034}

\bibitem[Har77]{Har}
Robin Hartshorne, \emph{Algebraic geometry}, Springer-Verlag, New York, 1977,
  Graduate Texts in Mathematics, No. 52. \MR{MR0463157 (57 \#3116)}

\bibitem[Lam70]{Laman}
G.~Laman, \emph{On graphs and rigidity of plane skeletal structures}, J. Engrg.
  Math. \textbf{4} (1970), 331--340. \MR{42 \#4430}

\bibitem[Mar03]{GGV}
Jeremy~L. Martin, \emph{Geometry of graph varieties}, Trans. Amer. Math. Soc.
  \textbf{355} (2003), no.~10, 4151--4169 (electronic). \MR{MR1990580
  (2005b:05074)}

\bibitem[Mar05]{Picspace}
\bysame, \emph{On the topology of graph picture spaces}, Adv. Math.
  \textbf{191} (2005), no.~2, 312--338. \MR{MR2103216 (2005i:05049)}

\bibitem[Mar06]{Slopes}
\bysame, \emph{The slopes determined by {$n$} points in the plane}, Duke Math.
  J. \textbf{131} (2006), no.~1, 119--165. \MR{MR2219238 (2007e:05041)}

\bibitem[Oxl92]{Oxley}
James~G. Oxley, \emph{Matroid theory}, Oxford Science Publications, The
  Clarendon Press Oxford University Press, New York, 1992. \MR{MR1207587
  (94d:05033)}

\bibitem[Ser00]{Ser}
Brigitte Servatius, \emph{On the rigidity of {R}amanujan graphs}, Ann. Univ.
  Sci. Budapest. E\"otv\"os Sect. Math. \textbf{43} (2000), 165--170 (2001).
  \MR{MR1847877 (2002g:05158)}

\bibitem[Sta07]{HA}
Richard~P. Stanley, \emph{An introduction to hyperplane arrangements},
  Geometric combinatorics, IAS/Park City Math. Ser., vol.~13, Amer. Math. Soc.,
  Providence, RI, 2007, pp.~389--496. \MR{MR2383131}

\end{thebibliography}
\end{document}